\documentclass[11pt,oneside]{amsart}
\usepackage{amsmath,ifthen, amsfonts, amssymb, srcltx,
   amsopn, textcomp}

\usepackage{pdfsync}

\usepackage{hyperref}
\usepackage{graphicx}

\usepackage[small]{caption}

\newcommand{\showcomments}{yes}
\renewcommand{\showcomments}{no}

\newcommand{\hidetodo}[1]
{\ifthenelse{\equal{\showcomments}{yes}}%
{#1}
}

\newsavebox{\commentbox}
\newenvironment{com}%
{\ifthenelse{\equal{\showcomments}{yes}}%
{\footnotemark
        \begin{lrbox}{\commentbox}
        \begin{minipage}[t]{1.25in}\raggedright\sffamily\tiny
        \footnotemark[\arabic{footnote}]}
{\begin{lrbox}{\commentbox}}}%
{\ifthenelse{\equal{\showcomments}{yes}}%
{\end{minipage}\end{lrbox}\marginpar{\usebox{\commentbox}}}
{\end{lrbox}}}

\newtheorem{thm}{Theorem}[section]
\newtheorem{lem}[thm]{Lemma}

\newtheorem{cor}[thm]{Corollary}

\newtheorem{prop}[thm]{Proposition}

\theoremstyle{definition}
\newtheorem{defn}[thm]{Definition}
\newtheorem{rem}[thm]{Remark}
\newtheorem{exmp}[thm]{Example}

\newtheorem{prob}[thm]{Problem}

\newcommand{\neb}{\mathcal N}

\DeclareMathOperator{\stabilizer}{Stabilizer}

\newcommand{\hull}{\text{\sf hull}}

\newcommand{\dist}{\textup{\textsf{d}}}

\newcommand{\field}[1]{\mathbb{#1}}
\newcommand{\integers}{\ensuremath{\field{Z}}}

\newcommand{\naturals}{\ensuremath{\field{N}}}
\newcommand{\reals}{\ensuremath{\field{R}}}

\newcommand{\hyperbolic}{\ensuremath{\field{H}}}

\DeclareMathOperator{\Isom}{Isom}

\renewcommand{\angle}{\sphericalangle}

\DeclareMathOperator{\diameter}{\text{diam}}

\newcommand{\hh}{\textsf{h}}
\newcommand{\vv}{\textsf{v}}

\begin{document}

\title{aTmenability of some graphs of groups with cyclic edge groups}
\author{Mathieu Carette}
\email{mathieu.carette@uclouvain.be}
\author[D.~T.~Wise]{Daniel T. Wise}
           \address{Dept. of Math. \& Stats.\\
                    McGill Univ. \\
                    Montreal, QC, Canada H3A 0B9 }
           \email{wise@math.mcgill.ca}
\author{Daniel Woodhouse}
\email{daniel.woodhouse@mail.mcgill.ca}
\date{\today}
\thanks{Research supported by NSERC}

\maketitle

\begin{com}
{\bf \normalsize COMMENTS\\}
ARE\\
SHOWING!\\
\end{com}

\begin{abstract}
We show that certain graphs of groups with cyclic edge groups are aTmenable.
In particular, this holds when each vertex group is either virtually special or acts properly and semisimply on $\hyperbolic^n$.
% Kleinian.
\end{abstract}

\section{Introduction}

\begin{com}
For uncountable (discrete) groups, the two definitions of a-T-menability are no longer equivalent.

1) The definition by means of existence of a proper isometric action on a Hilbert space, FORCES the group to be countable. Indeed, if $\alpha$ is such an action, then $F_n :=\{g\in G: d(0,\alpha_g(0)<n \}$ is finite (by properness), and the union of all $F_n$'s is $G$ itself.

2) If you take as the definition the existence of a unitary representation almost having invariant vectors and having coefficients vanishing at infinity, then it is true that a group is a-T-menable if and only if all its countable subgroups are.

So, the additive group of the reals, with the discrete topology, satisfies the latter but not the former definition.
Consequently, our main result applies to an uncountable group using definition 2, since definitions 1 and 2 are equivalent for uncountable groups.

We thank Alain Valette for pointing this out.
\end{com}

A group $G$ is \emph{aTmenable} or has the \emph{Haagerup property} if $G$ admits a metrically proper action  on a Hilbert space by affine isometries.
Gromov suggested the term aTmenable to highlight that aTmenability is both a generalization of amenability and a strong negation of Property~$(T)$ \cite{Gromov93}.
Throughout this paper, we will use the following equivalent definition due to Cherix, Martin and Valette: a discrete group $G$ is \emph{aTmenable} if it acts metrically properly on a space with measured walls (see Section~\ref{sec:MeasuredWallspaces}).  For other equivalent definitions of aTmenability, we refer to \cite{CCJJV01,CherixMartinValette04,ChatterjiDrtutuHaglund10}.

If a group $G$ splits as a graph of groups with aTmenable vertex groups and finite edge groups then $G$ is aTmenable \cite[Theorem 6.2.8]{CCJJV01}. In contrast, aTmenability can fail
for a graph of groups whose vertex and edge groups are isomorphic to $\integers^2$. For instance,
 as shown in \cite{CaretteQItoAmenable}, the following $\integers^2$-by-$F_2$ group  is not aTmenable:
 \[G = \integers^2 \rtimes \left\langle \left( \begin{array}[pos]{cc} 1 & 2 \\ 0 & 1 \end{array} \right),
\left( \begin{array}[pos]{cc} 1 & 0 \\ 2 & 1 \end{array} \right) \right\rangle \]
The following problem was raised in \cite[7.3.2]{CCJJV01}.
\begin{prob}
Let $G$ split as a graph of groups.
Suppose each vertex group is aTmenable and each edge group is cyclic.
Is $G$ aTmenable?
\end{prob}

Our main result, Theorem~\ref{thm:main}, solves a special case of this problem.
As its statement is technical we highlight the following attractive consequence:

\begin{thm}\label{thm:typical application}
Let $G$ split as a countable graph of groups where each edge group is infinite cyclic,
and each vertex group either acts properly and semisimply on $\hyperbolic^n$%a discrete Kleinian group
or is a virtually special group.
Then $G$ is aTmenable.
\end{thm}

\section{Measured Wallspaces and aTmenability} \label{sec:MeasuredWallspaces}

\begin{defn}[Measured Wallspace]
Let $X$ be a set.
A \emph{wall} is a partition of $X$ into two disjoint sets.
A wall $\Lambda = H \sqcup H^c$ \emph{s\hspace{0.05em}e\hspace{0.05em}p\hspace{0.05em}a\hspace{0.05em}r\hspace{0.05em}a\hspace{0.05em}t\hspace{0.05em}e\hspace{0.05em}s} $x$ and $y$ if $H \cap \{x,y\} $ is a singleton.
Let $\mathcal{W}$ be a set of walls in $X$.
Let $\mu$ be a measure on a $\sigma$-algebra $\mathcal{B}$ of  subsets of $\mathcal{W}$.
Let $\omega(x,y) \subseteq \mathcal{W}$ be the subset of walls separating $x$ and $y$.
Then $(X, \mathcal{W}, \mathcal{B}, \mu)$ is a \emph{measured wallspace} if $\omega(x,y) \in \mathcal{B}$ for each $x,y \in X$.

The function $x,y \mapsto \mu(\omega(x,y))$ is a pseudo-metric on $X$ called the \emph{wall metric} which we denote by $\#(x,y)$.
For subsets $A, B \subseteq X$ a wall $\Lambda = H \sqcup H^c$ \emph{separates} $A$ and $B$ if either $A \subseteq H$ and $B \subseteq H^c$, or $A \subseteq H^c$ and $B \subseteq H$.
Let $\omega(A,B)$ denote the set of walls separating $A$ and $B$.
If $A, B$ are countable subsets, then $\omega(A,B) \in \mathcal{B}$ since $\omega(A, B) = \bigcap_{a \in A, b \in B} \omega(a,b) \in \mathcal{B}$.
Let $\#(A,B) = \mu(\omega(A,B))$.
%
%We also define $\widehat{\#}(A, B) = \min\{\#(a,b) \mid a \in A, b \in B\}$.
%Note that $\widehat{\#}(A, B) \geq \widehat(A,B)$ and that if $A$ and $B$ are singletons, there is equality.
\end{defn}

\begin{exmp}[$\reals$ as a measured wallspace]
\label{exmp:the Line}
  Let $\mathcal{W}_\reals$ be the collection of walls of $\reals$ where each $W\in \mathcal{W}_\reals$ is of the form
  $(-\infty,r] \sqcup (r,\infty)$ for some $r\in \reals$.
  There is a bijection $\mathcal{W}_\reals\rightarrow \reals$ so we pull back the Lebesgue measure $\lambda$ and Borel algebra $\mathcal{B}_\reals$ to obtain a measured wallspace:
  $(\reals, \mathcal{W}_\reals, \mathcal{B}_\reals, \lambda)$.

 Let $\mathbf{R}$ be the cube complex homeomorphic to $\reals$ with 0-skeleton $\integers$.
  The measured wallspace $(\mathbb{R}, \mathcal{W}_\reals, \mathcal{B}_\reals, \lambda)$ is isomorphic mod-zero (see Definition~\ref{defn:isomorphism}) to the continuous measured wallspace on $(\mathbf{R},\mathcal{W}, \mathcal{B},\mu)$ described in Example~\ref{exmp:CAT0CubeComplexes} and also with the measured wallspace structure on $\hyperbolic^n$ in Examples~\ref{exmp:hyperbolicSpace} with $\reals=\hyperbolic^1$.
\end{exmp}

\begin{exmp}[Cube complexes as measured wallspaces] \label{exmp:CAT0CubeComplexes}
A CAT(0) cube complex $\widetilde{X}$ has two natural measured wallspace structures: the discrete and continuous.
 The walls of the \emph{discrete} wallspace correspond to the hyperplanes in $\widetilde{X}$, and we use the counting measure.
This is a measured wallspace since any two points in $\widetilde{X}$ are separated by a finite number of hyperplanes.

The \emph{continuous} measured wallspace structure will be more appropriate for our purposes.
Each hyperplane $\Lambda$ of $\widetilde{X}$ lies in a convex subcomplex called its \emph{carrier}, consisting of all cubes intersecting $\Lambda$.
The carrier is isometric to $\Lambda \times [-1,1]$ where $\Lambda$ corresponds to $\Lambda \times \{ 0\}$.
 For each hyperplane $\Lambda$, and $\alpha \in (-1,1)$ we let each $\Lambda \times \{ \alpha \}$ be a wall.
 We give the set of walls $\{\Lambda \times \alpha : \alpha \in (-1,1)  \}$ the Lesbesgue measure on $(-1,1)$.
 Note that the measure of the set of walls separating two $0$-cubes is identical in the discrete and continuous wallspaces.
\end{exmp}

\begin{exmp}[$\hyperbolic^n$ as a measured wallspace] \label{exmp:hyperbolicSpace}
We refer to \cite{Robertson98} for the full details of the construction outlined here.
Let $\hyperbolic^n$ denote $n$-dimensional real hyperbolic space.
Let $\mathcal{H}$ denote the set of all totally geodesic codimension-1 hyperplanes.
Let $\Lambda_0 \in \mathcal{H}$.
Every totally geodesic codimension-1 hyperplane in $\hyperbolic^n$ is in the orbit $\Isom(\hyperbolic^n) \Lambda_0$.
Identifying  $\mathcal{H}$ with the set of left cosets $\Isom(\hyperbolic^n) / \stabilizer(\Lambda_0)$,
we can endow $\mathcal{H}$ with the quotient topology.
The Haar measure $\mu$ for $Isom(\hyperbolic^n)$ passes to an invariant measure on the quotient since both $\Isom(\hyperbolic^n)$ and $\stabilizer(\Lambda_0)$ are unimodular.
It is immediate that $\mu$ is defined on the Borel sets $\mathcal{B}_{\hyperbolic}$ of $\mathcal{H}$.
As shown in \cite{Robertson98, CherixMartinValette04}, after scaling $\mu$ we have:
\begin{equation} \label{eq:hyperbolicDistance}
\#(x,y) = \dist_\hyperbolic(x,y)
\end{equation}
Metric properness of a group acting on $(\hyperbolic^n, \mathcal{H}, \mathcal{B}_\hyperbolic, \mu)$ is then equivalent to metric properness of the action on $\mathbb{H}^n$ itself.
\end{exmp}

Amenable groups may have quite different measured wallspace structures from the examples above, and we have not found a way to apply the method of this paper to arbitrary graphs of groups with amenable vertex groups and cyclic edge groups.

\begin{exmp} \label{exmp:amenableGroups}
Let $G$ be a countable amenable group.
Let $\{ K_n \}_{n=1}^{\infty}$ be an increasing, exhaustive sequence of finite subsets of $G$.
One way to define the amenability of $G$ is to require that there is a sequence of \emph{Folner} sets $\{ A_n \}_{n=1}^{\infty}$ such that $A_n$ is a finite subset of $G$, and such that for each $g \in K_n$ we have:
\[
\frac{|g A_n \triangle A_n|}{|A_n|} < 2^{-n}
\]

\noindent For each $n$, the partition $A_n \sqcup A_n^c$ is a wall of $G$ that is assigned a weight $n / |A_n|$.
A measured wallspace is obtained by including all $G$-translates of these partitions.
This is a measured wallspace and $G$ acts metrically properly on it \cite{CherixMartinValette04}.
  \end{exmp}

\begin{rem}
A primary difficulty in generalizing the results of this paper to arbitrary graphs of aTmenable groups with cyclic edge groups arises from the rich diversity of aTmenable structures for $\integers$.
An interesting test case would be an amalgamated free product
$A*_\integers B$ where the aTmenable structure for $\integers$ on the left
and right are highly unrelated.
Other difficulties related to the ``dispersal'' property discussed in Section~\ref{sec:dispersal}, and dealing with appropriate fundamental domains.
\end{rem}

\begin{defn} \begin{com} I have added this definition, that was otherwise absent. \end{com}
 Let $(X_1, \mathcal{W}_1, \mathcal{B}_1, \mu_1)$ and $(X_2, \mathcal{W}_2, \mathcal{B}_2, \mu_2)$ be measured wallspaces, and let $p_i : X_1 \times X_2 \rightarrow X_i$ be the projection.
 Then we can define the \emph{product} of these two wallspaces, as in~\cite{CherixMartinValette04}, with $X = X_1 \times X_2$, $\mathcal{W} = p^{-1}_1(\mathcal{W}_1) \sqcup p^{-1}_2(\mathcal{W}_2)$, $\mathcal{B}$ the $\sigma$-algebra on $\mathcal{W}$ given by $\mathcal{B}_1 \sqcup \mathcal{B}_2$, and $\mu$ the unique measure on $\mathcal{W}$ such that the restiction to $\mathcal{B}_i$ is $\mu_i$.
\end{defn}

The following is proven in \cite{CherixMartinValette04}:
\begin{lem}\label{lem:prop measured wallspace action gives atmenable}
 Suppose $G$ admits a metrically proper action on a measured wallspace.
Then $G$ is aTmenable.
\end{lem}
The following is proven in \cite[Prop~2.5(1)]{Jolissaint00}.
\begin{lem}\label{lem:aTmen by amen}
Suppose $G$ is $N$-by-$Q$ with $N$ aTmenable and $Q$ amenable. Then $G$  is aTmenable.
\end{lem}

\section{Dispersed subgroup relative to an action} \label{sec:dispersal}

\begin{defn}[Dispersed]
Let $(X, \mathcal{W}, \mathcal{B}, \mu)$ be a measured wallspace.
Let $\{Y_i\}_{i\in \mathbb{N}}$ be a collection of countable subsets of $X$.
Then  $\{ Y_i \}_{i\in \mathbb{N}}$ is \emph{dispersed} if for each $d>0$ there exists $n>0$ such that there do not exist  $i_1 < \cdots <i_n$, such that
$\#(Y_{i_p}, Y_{i_q})\leq d$ for all $p \neq q$.

Let $G$ act on the measured wallspace $(X,\mathcal W,\mathcal{B},\mu)$.
Let $H$ be a subgroup of $G$, and let $\{g_1, g_2, \ldots \}$ be an enumeration of left-coset representatives of $H$ in $G$.
A subgroup $H \leqslant G$ is \emph{dispersed relative to $x \in X$} if $\{ g_iHx \}_{i\in \mathbb{N}}$ is dispersed.
\end{defn}

\begin{rem} \label{rem:basepointInvariance}
 Dispersal of a subgroup is relative to a choice of basepoint $x \in X$, and it is unclear if dispersal is invariant of the choice.

Let $x, y \in X$ and suppose that $H$ is dispersed relative to $x$.
Let $\{g_i\}_{i\in \naturals}$ be a sequence of elements of $G$ such that $g_i H$ is an enumeration of all the cosets of $H$.
As $H$ is dispersed relative to $x$, we know $\#(Hx, g_iHx) \rightarrow \infty$ as $i \rightarrow \infty$.
As $(X, \mathcal{W}, \mathcal{B}, \mu)$ is a wallspace we know $\#(x,y) =d$ is finite, so we can then deduce that $\#(Hx, Hy) \leq d$ and $\#(g_iHx, g_iHy) \leq d$.
Furthermore, $\#(y, g_iy) \geq \#(x, g_ix) - \#(x,y) - \#(g_ix, g_iy) \geq \#(x, g_ix) -2d$.

The principal difficulty of showing invariance of the basepoint is that the triangle inequality cannot be applied to $\#(Hy, g_i Hy)$.
Indeed, $\omega(Hx, g_iHx)$ can be partitioned as $\Omega_1^i \sqcup \Omega_2^i \sqcup \Omega_3^i$ where:\begin{com}
     CORRECTION - This is confusing because previously $H$ and $H^c$ were used to denote the halfspaces.
	\end{com}
\[ \Omega_1^i = \{ \Lambda \in \omega(Hy, g_iHy) \cap \omega(Hx, g_iHx) \}\]
\[ \Omega_2^i = \{ \Lambda \in \omega(Hx, g_iHx) \mid \textrm{$Hy \cup g_iHy$ is contained in a single halfspace of $\Lambda$} \} \]
\[ \Omega_3^i = \{ \Lambda \in \omega(Hx, g_iHx) \mid \textrm{$\Lambda$ non-trivially partitions either $Hy$ or $g_iHy$} \}\]
To show $H$ is dispersed relative to $y$ it suffices to show that $\mu(\Omega_1^i) \rightarrow \infty$ as $i \rightarrow \infty$. \begin{com} CORRECTION 2 - I reversed the implication. \end{com}
 We know $\mu(\Omega_2^i)$ is bounded for all $i$ since walls in $\Omega_2^i$ either separate $g_iy$ from $g_ix$, or separate $x$ from $y$.
Therefore, showing the dispersal of $H$ relative to $y$ can be viewed as a question of estimating $\mu(\Omega_3^i)$.
An upper bound on $\mu(\Omega_3^i)$ for all $i$, would imply that $\mu(\Omega_1^i)$ diverges to infinity.
Hence dispersal of $H$ relative to $x$ would imply dispersal relative to $y$.

In Proposition~\ref{prop:basepointInvarianceCAT0} we prove that $\Omega_3^i$ is bounded for CAT(0) cube complexes in our case of interest.
We expect this is true for $\hyperbolic^n$ as well.
\end{rem}

For a subcomplex $Y \subseteq X$ of a CAT(0) cube complex, let $\neb_m(\widetilde Y)$ denote the $m$-neighborhood of $\widetilde Y$.
The \emph{convex hull} of a subset $S$ of a CAT(0) cube complex $X$ is the smallest convex subcomplex $\hull(S)$ containing $S$.
We refer to \cite[Lem~13.15]{HaglundWiseSpecial} for the following:

\begin{lem}\label{lem:r thickening}
	Let $Y$ be a convex subcomplex of a finite dimensional CAT(0) cube complex $X$.
	For each $r\geq 0$ there exists a convex subcomplex $ Y^{+r}$ such that $\neb_r( Y)\subset Y^{+r} \subset \neb_s( Y)$
	for some $s = \dim(X) r$.
\end{lem}

We refer to the convex subcomplex $Y^{+r}$ as the \emph{$r$-thickening} of $Y$.

\begin{prop} \label{prop:basepointInvarianceCAT0}
 Suppose $G$ acts properly and cocompactly on a CAT(0) cube complex $X$.
 Let $H \leqslant G$ be a subgroup that acts cocompactly on a nonempty convex subcomplex $Y \subseteq X$.
 If $H$ is dispersed relative to a $0$-cube $x \in X$, then $H$ is dispersed relative to any $0$-cube $x'\in X$.
\end{prop}

%We use the notation $\neb_k(Y)$ for the $k$-neighborhood of $Y$ in $X$.
%The \emph{$m$-fold cubical thickening}  of a convex subcomplex $Y \subseteq X$ is a convex subcomplex $Y^{+m}$ with the property that $\neb_m(Y) \subseteq Y^{+m} \neb_n(Y)$ for some $n$.
%Moreover, if $Y$ is $H$-invariant, then so is $Y^{+m}$.
%We refer to  \cite{HaglundWiseSpecial}.

\begin{proof}
Without loss of generality, we may assume that $Y = \hull(Hx)$.
Indeed, we may assume that $x \in Y$ be replacing $Y$ with $Y^{+\ell}$ for some $\ell$, and thereafter that $Y = \hull(Hx)$ by replacing $Y$ with the convex hull of $Hx$ in $Y$.

Let $Y' = \hull(Hx')$.
The hyperplanes intersecting $\hull(Hx')$ are precisely the hyperplanes that non-trivially partition $Hx'$.
The hyperplanes that separate $Hx$ and $gHx$ are precisely the hyperplanes that separate $Y$ and $gY$, or equivalently the hyperplanes on a shortest geodesic between $Y$ and $gY$.
Following Remark~\ref{rem:basepointInvariance}, it suffices to bound $\Omega_3$; that is to bound the number of hyperplanes separating $Y, gY$, but intersecting $Y'$ or $gY'$.
By symmetry it suffices to bound those intersecting $Y'$.

Observe that $Y'\subset Y^{+m}$ for some $m$.
Choose a shortest geodesic $\sigma$ between $gY$ and $Y^{+m}$,
and orient $\sigma$ so that it ends at a point $v \in  Y^{+m}$.
Let $u\in Hx$ be a point with $v\in \neb_m(u)$, and let $\sigma'$ be a geodesic from $v$ to $u$.
Observe that any hyperplane $\Lambda$ separating $Y,gY$ and intersecting $Y'$ must intersect $\sigma'$.
Indeed, if $\Lambda$ separates $Y,gY$ then $\Lambda$ must be crossed by the concatenation $\sigma\sigma'$.
However, each hyperplane crossed by $\sigma$ is disjoint from $Y^{+m}$ and hence from $Y'$.
Thus $\Lambda$ must be crossed by $\sigma'$, and so the number of such hyperplanes is bounded by $|\sigma'|\leq m$.
\end{proof}

\begin{com}
	REMOVED: Therefore, when talking about an action on a CAT(0) cube complex $X$, we omit any assertion of which basepoint the dispersal is relative to and simply say that a subgroup is dispersed relative to the action on $X$.
\end{com}

%\begin{exmp} \label{exmp:simpleExample}
%Let $G$ be a finitely generated abelian group, and let $H$ be a subgroup.
% There exists an action of $G$ on a CAT(0) cube complex $X$ such that $H$ is dispersed with respect to either the discrete or %continuous measured wallspace structures of Example~\ref{exmp:CAT0CubeComplexes}.
% We ensure that $H$ is dispersed by choosing codimension 1 subgroups such that $H$ is commensurable with the intersection of %a subset of the hyperplane stabilizers.
% In fact, $H$ is dispersed if and only if $H$ is commensurable with the intersection of hyperplane stabilizers. \begin{com} %TO DO:
%
%  Comment about next lemma.
%\end{com}.
%\end{exmp}

\begin{lem} \label{lem:dispersalTransitivity}
 Let $G_3 \leqslant G_2 \leqslant G_1$, and let $(X,\mathcal{W},\mathcal{B}, \mu)$ be a measured wallspace with a $G_1$ action.
 If $G_2 \leqslant G_1$ is dispersed relative to $x \in X$, and $G_3 \leqslant G_2$ is dispersed relative to $x \in X$ then $G_3 \leqslant G_1$ is dispersed relative to $x \in X$.

 Conversely, if $G_3 \leqslant G_1$ is dispersed relative to $x\in X$, then $G_3 \leqslant G_2$ is dispersed relative to $x \in X$.
\end{lem}

 \begin{proof}
  In the first case, suppose that $G_3$ is not a dispersed subgroup of $G_1$ relative to $x$.
  Then there exists a sequence $\{ g_i \}_{i\in \naturals}$ such that $g_i G_3 \neq g_j G_3$ for $i \neq j$, and $\#(g_i G_3 x , G_3 x) < D$ for all $i \in \naturals$.\begin{com}
  	{CORRECTION: Added the $\#$.}
  \end{com}
  If there are infinitely many $g_i G_3$ contained in a single left coset of $G_2$, then this contradicts the dispersal of $G_3$ in $G_2$ relative to $x$.
  Otherwise, there must be infinitely many $g_iG_3$ contained in infinitely many distinct $G_2$ cosets, which contradicts the dispersal of $G_2$ in $G_1$ relative to $x$.

  The converse case follows from the observation that the $G_2$ cosets of $G_3$ are a subset of the $G_1$ cosets of $G_3$.
 \end{proof}

%\begin{lem}
%	Let $G$ act on a measured wallspace $(X, \mathcal{W}, \mathcal{B}, \mu)$.
%	Let $H \leq G$ and let $G' \leq G$ be a finite index subgroup and let $H' = (G' \cap H)$.
%	Suppose that $H' \leq G'$ is dispersed with respect to $x \in X$.
%	Then $H \leq G$ is dispersed relative to $x \in X$.
%\end{lem}
%
%\begin{proof}
%   Suppose that $H$ is not dispersed relative to $x$.
%   Then there exists a $D > 0$, and arbitrarily large collections of $g_1, \ldots ,g_m \in G$ such that $\#(g_iHx, g_jHx) < D$ for all $1\leq i < j \leq m$.
%   As $G' \leq G$ is finite index, there are finitely many $G'$ orbits of left $H$-cosets.
%   Therefore, we may assume, that after translating by elements in $G$, that our arbitrarily large collections $g_1, \ldots,g_m$ are such that $g_1H, \ldots, g_m H$ all belong in the orbit $G' H$.
%   Therefore, there exists $g_1', \ldots, g_m' \in G'$ such that $g_i'H' \subseteq g_iH$, which implies that $\#(g_i'H'x, g_j'H'x) < D$ for all $1 \leq i < j \leq m$.
%   Since $m$ can be arbitrarily large, this contradicts the dispersal of $H'\leq G'$ relative to $x\in X$.	
%\end{proof}

\begin{lem} \label{lem:dispersalFiniteIndex}
	Let $G$ act on a finite dimensional CAT(0) cube complex $X$.
	Let $H \leq G$.
	Let $G' \leq G$ be a finite index subgroup, and let $H' = (G' \cap H)$.
	Suppose that $H' \leq G'$ is dispersed with respect to $x \in X$.
	Then $H \leq G$ is dispersed relative to $x \in X$.
\end{lem}

\begin{proof}
	The claim follows by showing that for all $k >0$ there exists $K>0$ such that if $g_aH'$ and $g_bH'$ are distinct left cosets, then $\#(g_aH'x,g_bH'x)>K$ implies that $\#(g_aHx, g_bHx)>k$.
	Indeed, first observe by Lemma~\ref{lem:dispersalTransitivity} that if $H' \leqslant G'$ is dispersed relative to some $0$-cube, then so is $H' \leqslant G$.
	Now if $g_1 H , \ldots, g_m H$ are distinct cosets with $m = m(K)$, then so are $g_1 H', \ldots, g_m H'$, so there exists $\#(g_aH'x, g_bH'x) > K$, and hence $\#(g_aHx, g_bHx)>k$.
	
	Suppose that $\#(g_aH'x, g_bH'x) = k$, for $k$ sufficiently large.
	Let $Y_a = \hull(g_aH'x)$, and let $Y_b = \hull(g_bH'x)$.
	% denote the subcomplex consisting of the intersection of all (small) halfspaces containing $g_aHx$ but not $g_bH'x'$.
	%Likewise let $Y_b$ denote the subcomplex consisting of the intersection of all halfspaces containing $g_bH'x$ but not $g_aH'x$.
	Note that $Y_a$ and $Y_b$ are separated by $k$ walls.
	
	There exists $r = r(H, H',x)$ such that $Hx \subseteq \neb_r(H'x)$, and therefore $Hx \subseteq (Y_a)^{+r}$.
	By Lemma~\ref{lem:r thickening} we observe that $(Y_a)^{+r} \subseteq \neb_s(Y_a)$ for $s =\dim(X)r$.	
	We thus find that the number of walls separating $g_aHx$ and $g_bHx$ is at least $k-2s$.
	Indeed, if $\sigma$ is a combinatorial geodesic joining $(Y_a)^{+r}$ and $(Y_b)^{+r}$, then it has length at least $k-2s$ and the walls intersecting $\sigma$ separate the convex subcomplexes $(Y_a)^{+r}$ and $(Y_b)^{+r}$.
	
	In conclusion, given $K>0$, by letting $k = K+2s = K + 2r\dim(X)$ we ensure  $\#(g_aH'x,g_bH'x)>K$ implies $\#(g_aHx, g_bHx)>k$.
	\end{proof}

 Although Lemma~\ref{lem:dispersalFiniteIndex} only applies in the setting when the measured wallspace is a finite dimensional CAT(0) cube complex, the following converse applies to general measured wallspaces.

\begin{lem} \label{lem:IntersectionDispersal}
   Let $G$ act on a measured wallspace $(X, \mathcal{W}, \mathcal{B}, \mu)$.
   Suppose that $H \leqslant G$ is dispersed relative to $x \in X$.
   Let $G' \leq G$, and $H' = H \cap G'$.
   Then $H' \leq G'$ is dispersed relative to $x \in X$.	
\end{lem}

\begin{proof}
	Suppose that $H' \leq G'$ is not dispersed relative to $x \in X$.	
	Then there exists $D > 0$ and arbitrarily large collections of $g_1, \ldots, g_m \in G'$ such that $g_iH' \neq g_jH'$ and $\#(g_iH'x, g_jH'x) <D$ for $1 \leq i < j \leq m$.
	Then $g_i H \neq g_j H$ for $i \neq j$, otherwise $g^{-1}_jg_i \in  G' \cap H =H'$, which would imply  that $g_iH'x = g_jH'x$.
	Thus there are arbitrarily large collections of left cosets $g_1 H, \ldots, g_m H$ such that $\#(g_iH'x, g_jH'x) <D$ for $1 \leq i < j \leq m$.
\end{proof}

\begin{prop} \label{prop:abelianDispersal}
 Let $G$ be a finitely generated abelian group, and let $H$ be a subgroup of $G$.
 Then $G$ acts on a CAT(0) cube complex $X$ such that $H$ is dispersed relative to any $0$-cube $x$ in $X$.% relative to any $0$-cube in $X$.
\end{prop}

\begin{proof}
 Let $G$ be a rank $n$ abelian group, and $H$ a rank $k$ subgroup.
 Fix a proper action of $G$ on $\mathbb{R}^n$.
 Choose $g_{k+1}, \ldots, g_n \in G$ such that $\langle H, g_{k+1}, \ldots, g_n \rangle$ is commensurable with $G$.
 By Lemma~\ref{lem:dispersalTransitivity}, and the fact that finite index subgroups are always dispersed relative to any action, we can assume that $G = \langle H, g_{k+1}, \ldots, g_n \rangle$.

 Let $\Lambda_i \subseteq \mathbb{R}^n$ be the hyperplane stabilized by $\langle H, g_{k+1}, \ldots g_{i-1}, g_{i+1}, \ldots, g_n \rangle$, for $k_1 \leq i \leq n$.
 As each $\Lambda_i$ is a codimension-1 hyperplane, $\Lambda_i$ defines a wall in $\mathbb{R}^n$.
 Taking the union $\mathcal{W}$ of all $\Lambda_i$ and their $G$-translates, we obtain a wallspace $(\mathbb{R}, \mathcal{W})$.
 Let $X$ be the dual cube complex $C(\mathbb{R}^n, \mathcal{W})$.
 Observe that $C(\mathbb{R}^n, \mathcal{W}) = \mathbf{R}^{n-k}$, there is only one $0$-cube orbit, and that the stabilizer in $G$ of any $0$-cube is precisely $H$.
 Therefore, the dispersal of $H$ is equivalent to the properness of $\mathbf{R}^{n-k}$.
 \end{proof}

\begin{lem} \label{lem:cocompactDispersal}
Let $G$ act properly and cocompactly on a CAT(0) cube complex $X$.
Let $H$ be a subgroup that acts properly and cocompactly on a convex subcomplex $Y$.
Then $H$ is dispersed relative to  any $0$-cube $y \in Y$. \begin{com}CORRECTION -Reinserted `relative to the basepoint'.\end{com}
\end{lem}
\begin{proof}
By \cite[Lem~2.4]{WiseWoodhouse15} the subgroup $H$ has bounded packing in $G$.
Let $G \slash H = \{ g_1H, g_2H, \ldots \}$ be the collection of cosets of $H$.
\emph{Bounded packing} of $H$ in $G$ means that for all $d > 0$ there exists $n>0$ such that there do not exist $i_1 < \cdots < i_n$ with $\dist_G(g_{i_p}H, g_{i_q}H) \leq d$ for all $p \neq q$.
Since $G$ acts properly and cocompactly on $X$, and $H$ acts cocompactly on $Y$, we deduce the analogous statement that: for all $d > 0$ there exists $n>0$ such that there do not exist $i_1 < \cdots < i_n$ with $\dist_X(g_{i_p}Y, g_{i_q}Y) \leq d$ for all $p \neq q$.

Suppose that $\dist_X(g_1 Y, g_2Y) > 0$.
Let $\gamma$ be a combinatorial path in $X$ with endpoints on $g_1Y, g_2Y$, such that $|\gamma| = \dist_X(g_1Y, g_2Y)$
As $Y$ is convex, a hyperplane dual to $\gamma$ cannot intersect $g_1Y$ or $g_2Y$.
Therefore, $\#(g_1Y, g_2Y) = \dist_X(g_1Y, g_2Y)$.

Thus, the bounded packing of $H$ in $G$ implies the dispersion of $\{ g_i Y \}_{i\in \naturals}$, and therefore the dispersal of $H$ relative to any $0$-cube $y \in Y$.
\end{proof}

\begin{cor}
Let $G$ be a hyperbolic group that acts properly and cocompactly on a CAT(0) cube complex $X$ with either the continuous or discrete measured wallspace structure.
Let $H$ be a quasiconvex subgroup.
Then $H$ is dispersed relative to any $0$-cube $x$ in $X$. \begin{com} CORRECTION: added relative to any $0$-cube.
	\end{com}
\end{cor}
\begin{proof} By \cite{Haglund2006,SageevWiseCores}, there exists an $H$-cocompact convex subcomplex $Y\subset X$.
We may thus apply Lemma~\ref{lem:cocompactDispersal} and Proposition~\ref{prop:basepointInvarianceCAT0}.
\end{proof}

\begin{lem} \label{lem:angle of intersection}
	Let $\gamma \subseteq \mathbb{H}^n$ be a biinfinite geodesic, and let $\theta \in (0,\pi/2)$.
	The set of hyperplanes intersecting $\gamma$ at angle $<\theta$ is measurable.
\end{lem}

Let $[a,b]$ denote the set of points in the geodesic in $\mathbb{H}^n$ joining $a$ and $b$.

\begin{proof}
	Observe that the set of all walls intersecting $\gamma$ is measurable since it is equal to the ascending countable union of measurable sets $\omega(\gamma(-n), \gamma(n))$.
	Enumerate the dense, countable set $\{a_i\}_{i\in \mathbb{N}} \subseteq \gamma(\mathbb{Q})$.
	Let $a_i = \gamma(t_i)$.
	Let $\omega(\gamma(-\infty), a_i)$ denote the measurable set $\bigcup_{n < t_i}^{\infty} \omega(\gamma(n), a_i)$.
	Let $p_\gamma : \mathbb{H}^n \rightarrow \mathbb{H}^n$ be the projection map onto $\gamma$.
	
	Let $x_i = \gamma(t_i+1)$.
	Enumerate a countable subset of points $\{x_{ij}\}_{j\in \mathbb{N}} \subseteq p_{\gamma}^{-1}(x_i)$ that is dense in $p_{\gamma}^{-1}(x_i)$, and such that $d_{\mathcal{H}}(x_{ij}, x_i) \leq \tanh^{-1}(\tan(\theta) \sinh(1))$.
	The points $a_i,x_i,x_{ij}$ form a right angle triangle in $\mathbb{H}^n$ containing the unit segment $[a_i,x_i] \subseteq \gamma$, and hypotenuse $[a_i,x_{ij}]$ that meets $\gamma$ at an angle less than $\theta$.
	If $\Lambda \in \omega(\gamma(-\infty), a_i) \cap \omega(x_i,x_{ij})$, there is right angled triangle with vertices $\Lambda \cap \gamma, x_i,$ and $[x_i, x_{ij}] \cap \Lambda$.
	Let $\zeta$ be the angle at $\Lambda \cap \gamma$.
	Since the side adjacent to $\zeta$ contains $[a_i,x_i]$, and the side opposite $\zeta$ is contained in $[x_i, x_{ij}]$, we can deduce that $\zeta < \theta$.
    Therefore, $\omega(\gamma(-\infty), a_i) \cap \omega(x_i,x_{ij})$ only contains walls that intersect $[\gamma(-\infty), a_i]$ at an angle of less than $\theta$.
    Let

    \[ B_{\theta}^{\infty} = \bigcup_{i,j \in \mathbb{N}}\omega(\gamma(-\infty), a_i) \cap \omega(x_i,x_{ij}). \]

    \noindent Since $B_{\theta}^{\infty}$ is the countable union of measurable sets, $B_{\theta}^{\infty}$ is  a measurable set.

    Let $\Lambda$ be a hyperplane intersecting $\gamma$ at an angle less than $\theta$.
    Let $\alpha = \Lambda \cap \gamma$.
    Since $\{a_i\}_{i\in \mathbb{N}}$ is dense in $\gamma$ there exists a subsequence $\{ a_{i_k} \}_{k\in \mathbb{N}} \subseteq [\gamma(-\infty), \alpha]$ converging to $\alpha$.
    Note that for $k$ sufficiently large, $\Lambda$ will non-trivially intersect $p^{-1}(x_{i_k})$.
    Since $\{x_{i_kj}\}_{j\in\mathbb{N}}$ are also dense in $p_{\gamma}^{-1}(x_{i_k})$, for $k$ sufficiently large, there exist points $a_{i_k}$ and $x_{i_kj}$ such that $\Lambda \in \omega(\gamma(-\infty), a_{i_k}) \cap \omega(x_{i_kj}, p_{\gamma}(x_{i_kj}))$.
    Therefore $B_{\theta}^{\infty}$ is precisely the set of all hyperplanes intersecting $\gamma$ at an angle of less than $\theta$.
\end{proof}

\begin{lem} \label{lem:hyperbolicDispersal}
Let $G$ act discretely on $\hyperbolic^n$.
Consider the associated action on $(\hyperbolic^n, \mathcal{H}, \mathbb{B}_\hyperbolic, \mu)$.
A cyclic subgroup $H$ of $G$ is dispersed relative to a point in the axis $A$ of $H$.
\end{lem}
\begin{proof}

Let $\{ g_i H \}_{i\in \naturals}$ be the left-cosets of $H$.
To show that $H$ is dispersed relative to $a \in A$, it suffices to show that $\{g_i A\}_{i\in \naturals}$ is dispersed.
As $H$ has bounded packing in $G$, for each $d$ there exists $m$ such that in any cardinality $m$ subcollection $Q \subseteq \{g_i A \}_{i\in \naturals}$ there exists $g_p A, g_q A \in Q$ with $\dist_\hyperbolic(g_p A , g_q A) > d$.
It therefore suffices to show that there exists $d_0$ so that for biinfinite geodesics $A_1, A_2$ we have \begin{com}
	CORRECTION: Inserted \emph{so that}
\end{com}
\begin{equation} \label{eq:distanceBound}
\#(A_1, A_2) \geq \frac 1 2  \lfloor \dist_\hyperbolic(A_1, A_2 )  - 2d_0 \rfloor.
\end{equation}
Let $P$ be a geodesic between $A_1$ and $A_2$, realizing $\dist_{\hyperbolic}(A_1, A_2)$.
We will first show that for any wall $\Lambda$ cutting $P$, that does not separate $A_1, A_2$ either $\angle(\Lambda, P)$ is small or $\Lambda$ intersects $P$ close to $A_1$ or $A_2$.
We will secondly show that the measure of each of these subsets is a controlled part of the measure of the walls intersecting $P$.

 %We first establish that $\Lambda$ intersects close to each $A_i$.
 We first consider how the angle at which $\Lambda$ intersects $P$ relates to the distance of $\Lambda \cap P$ from $A_1$ and $A_2$. \begin{com}
 	{CORRECTION: I rewrote this line.}
 \end{com}
 Note that $P$ meets each $A_i$ at a right angle.
 Consider an ideal hyperbolic triangle with angles $\theta_0$, $\frac{\pi}{2}$, $0$, and let $d_0$ be the length of the finite base side.
 Thus $\sinh(d_0) = \cot(\theta_0)$.
 Any finite hyperbolic triangle with angles $\theta$, $\frac \pi 2$ meeting at a base of length $\geq d_0$, must satisfy $\theta < \theta_0$.
 Therefore, if $p \in P$ is at a distance more than $d_0$ from $A_i$, then a hyperplane $\Lambda$ in $\hyperbolic^n$ intersecting $p$ will not intersect $A_i$ provided $\angle(P, \Lambda) \geq \theta_0$.

%It is a fact from hyperbolic geometry that if an ideal right-angled hyperbolic triangle with angle $\theta_0$ at the non-zero adjacent angle, and $d_0$ the length of the finite side, then $\sinh(d_0) = \cot(\theta_0)$. \begin{com} Reference \end{com}
%Therefore, if $p \in P$ is at a distance more than $d_0$ from $A_i$, then a hyperplane $\Lambda$ in $\hyperbolic^n$ intersecting $p$ will not intersect $A_i$ provided $\angle(P, \Lambda) \geq \theta_0$.

We now estimate the measures.
Let $ab$ be a length 1 geodesic segment in $\hyperbolic^n$ and $c \in ab$.
Each hyperplane $\Lambda$ with $\Lambda \cap ab = c$ has a unique geodesic perpendicular to $\Lambda$ at $c$.
Therefore the set of all such hyperplanes maps bijectively to the open ball of radius $\frac{\pi}{2}$ in $\reals^{n-1}$, denoted $B_{\frac{\pi}{2}}$. \begin{com}
	I have removed the phrase \emph{maps homeomorphically} since we aren't talking about the topology yet. I have said that we actually have a bijection.
	PROBLEM: We need to say that $B_{\theta}$ is a measurable set in order for the rest of the argument to work.
\end{com}
Let $B_{\theta} \subset B_{\frac{\pi}{2}}$ denote the set of hyperplanes intersecting $ab$ at an angle less than $\theta$.
We may thus identify the set of all hyperplanes transversely intersecting $ab$ with $ab \times B_{\frac{\pi}{2}}$.
By Equation~\eqref{eq:hyperbolicDistance} we have $\#(a,b) = \dist_{\hyperbolic}(a,b)$.
Note that $B_{\frac{\pi}{2}} = \omega(a,b)$ is a measurable set.
By Lemma~\ref{lem:angle of intersection}, we deduce that $B_{\frac{\pi}{2}} - B_{\theta}$ is also measurable.
The measure of the set of hyperplanes intersecting the geodesic $ab$ at an angle of at least $\theta \in (0, \frac \pi 2)$ is therefore equal to
$ \mu( ab \times ( B_{ \frac{\pi}{2} } - B_{\theta}))$.
As $\mu(ab \times (B_{\frac{\pi}{2}} - B_{\theta})) \rightarrow \dist_{\hyperbolic}(a,b)$ as $\theta \rightarrow 0$, there exists $\theta_0 > 0$ such that $\mu(ab \times (B_{\frac{\pi}{2}} - B_{\theta_0})) \geq \frac 1 2 \dist_{\hyperbolic}(a,b) = \frac 1 2$.

Equation~\eqref{eq:distanceBound} follows by considering a maximal integral length subpath of $P$ that is at least $d_0$ from each $A_i$.
\end{proof}

A f.g. free abelian subgroup $A \leqslant G$ is \emph{highest} if $A$ does not have a finite index subgroup contained in a higher rank free abelian subgroup of $G$.
Note that if $G$ acts properly on a finite dimensional CAT(0) cube complex, then every free abelian subgroup of $G$ is contained in a highest subgroup of $G$.

We proved the following in \cite{WiseWoodhouse15}:
 \begin{lem}\label{lem:highest cubical torus} Let $X$ be a compact nonpositively curved cube complex.
 Let $A$ be a highest abelian subgroup of $\pi_1X$.
There is a compact based nonpositively curved cube complex $Y$ and a local isometry
$Y\rightarrow R$ with $\pi_1Y=A$.
\end{lem}

The subgroup $A_i\leqslant G$ is a \emph{virtual retract} if there is a finite index subgroup $G_i\leqslant G$ and a retraction $G_i\rightarrow A_i'$ to a finite index subgroup of $A_i'\leqslant A_i$.
$G$ is \emph{virtually special} if there is a finite index subgroup $G'\leqslant G$ such that $G'$ is isomorphic to a subgroup of a right-angled Artin group (raag).

\begin{prop} \label{prop:DispersedRAAGS}
	Let $G$ be finitely generated and virtually special.
	Let $Z_1, \ldots, Z_n$ be infinite cyclic subgroups of $G$.
	Then $G$ acts metrically properly on a CAT(0) cube complex $X$ such that $Z_i \leqslant G$ is dispersed relative to any $0$-cube in $X$.
\end{prop}

\begin{proof}
	Let $G' \leqslant G$ be an index $d$ subgroup such that $G'$ is special and hence $G' \hookrightarrow \pi_1 R$, where $R$ is a Salvetti complex for a right angled Artin group.
	Moreover, since $G'$ is finitely generated we may assume that $\pi_1 R$ is finitely generated and hence $R$ is compact.
	Let $Z_i' = Z_i \cap G'$.
	Let $A_i \leqslant \pi_1 R$ be a highest free abelian group containing $Z_i'$.
	By Lemma~\ref{lem:highest cubical torus} there exists a based local isometry $Y_i \rightarrow R$ from compact non-positively curved cube complex $Y_i$ such that $\pi_1 Y_i  = A_i$.
	Applying canonical completion and retraction~\cite{HaglundWiseSpecial} there is a degree $d_i$ cover $R_i \rightarrow R$ such that there is a retraction $r_i : \pi_1 R_i \rightarrow A_i$.
	By Lemma~\ref{prop:abelianDispersal} $A_i$ acts on a CAT(0) cube complex $X_i$ such that $Z_i' \leqslant A_i$ is dispersed relative to any choice of $0$-cube.
	Using the retraction $r_i$ there is an action of $\pi_1 R_i$ on $X_i$ that extends the action of $A_i$, and hence an action of $\pi_1 R$ on $(X_i)^{d_i}$ such that $Z_i' \leq A_i$ is dispersed relative to any $0$-cube in $(X_i)^{d_i}$. \begin{com}Do we want a lemma addressing the products of wallspaces?\end{com}
	So $G'$ acts on $(X_i)^{d_i}$, and by Lemma~\ref{lem:dispersalTransitivity} we deduce that $Z_i' \leq (G'\cap A_i)$ is dispersed relative to any $0$-cube in $(X_i)^{d_i}$.
	By Lemma~\ref{lem:cocompactDispersal} and Proposition~\ref{prop:basepointInvarianceCAT0} we know that $\pi_1 R$ acts metrically properly on $\widetilde{R}$ such that $A_i \leq \pi_1 R$ is dispersed relative to any $0$-cube.
	Therefore Lemma~\ref{lem:IntersectionDispersal} implies that $G'$ acts metrically properly on $\widetilde{R}$ such that $G'\cap A_i \leq G'$ is dispersed relative to any $0$-cube in $\widetilde{R}$.
	Thus, $G'$ acts diagonally on $\widetilde{R} \times \prod_{i=1}^n (X_i)^{d_i}$ such that the action is metrically proper and by Lemma~\ref{lem:dispersalTransitivity} $Z_i' \leq G'$ is dispersed relative to any $0$-cube.
	Finally $G$ acts metrically properly on $(\widetilde{R} \times \prod_{i=1}^n (X_i)^{d_i})^d$ such that $Z_i' \leq G'$ is dispersed relative to any $0$-cube in $(\widetilde{R} \times \prod_{i=1}^n (X_i)^{d_i})^d$, and therefore by Lemma~\ref{lem:dispersalFiniteIndex} $Z_i \leq G$ is dispersed relative to any $0$-cube in $(\widetilde{R} \times \prod_{i=1}^n (X_i)^{d_i})^d$.
\end{proof}

\section{Standard Probability Spaces}

\begin{defn} \label{defn:isomorphism}
An \emph{isomorphism} $f:(\Omega_1, \mathcal{B}_1, \mu_1)\rightarrow (\Omega_2, \mathcal{B}_2, \mu_2)$ between measure spaces
is an invertible map $f:\Omega_1\rightarrow \Omega_2$ such that $f(U) \in \mathcal{B}_2$ if and only if $U \in \mathcal{B}_1$, and moreover $\mu_1(U)=\mu_2(f(U))$.
We say $(\Omega_1,\mu_1)$ and $(\Omega_2,\mu_2)$ are isomorphic \emph{mod-zero}  if there is an isomorphism
between $(\Omega_1',\mu_1')$ and $(\Omega_2',\mu_2')$ where each $(\Omega_i',\mu_i')$ is obtained by removing a nullset.
\end{defn}

A probability space $(\Omega,\mathcal{B},\mu)$ is \emph{standard} if it is isomorphic mod-zero to $([0,1],\lambda)$.
We caution that standard probability spaces are sometimes defined in a more general fashion that permits them to contain atoms.
 We refer to \cite{delaRue93} for the following:
 \begin{thm}[Topological Characterization of Standard]
 \label{thm:topological characterization}
 Suppose $(\Omega,\mathcal{B},\mu)$ has no atoms.
 Then $(\Omega,\mathcal{B},\mu)$ is standard if and only if
 there exists a topology $\tau$ on $\Omega$ such that the following hold:
 \begin{enumerate}
  \item $(\Omega, \tau)$ is metrizable.
  \item $\mathcal{B}$ is the completion of the $\sigma$-algebra generated by $\tau$.
  \item $\sup \mu(K) =1 $ where the supremum is taken over all compact sets $K$.
 \end{enumerate}
\end{thm}

The following generalization of standard plays an important role in Section~\ref{sec:main}:

\begin{defn}[$\integers$-standard measure space]
Let $(\Omega, \mu)$ be a measure space with a $\integers$ action.
We say $(\Omega, \mu)$ is \emph{$\integers$-standard} if there is a $\integers$-equivariant mod-zero isomorphism to $(\reals, \varrho \lambda)$ where $\lambda$ is the Lebesgue measure, and $\varrho >0$.
\end{defn}

\begin{lem}[Fundamental Domain]
Let $(\Omega,\mu)$ be a measure space with a $\integers$-action.
Suppose it has a measurable fundamental domain $D\subset \Omega$ that is standard, after scaling the measure.
Then $(\Omega,\mu)$ is $\integers$-standard.
\end{lem}
\begin{proof}
The mod-zero isomorphism  $f:(D,\mu) \rightarrow ([0,\mu(D)],\lambda)$ extends to a $\integers$-equivariant
mod-zero isomorphism $f:(\Omega,\mu)\rightarrow (\reals,\lambda)$.
\end{proof}

\begin{exmp}
 The continuous measured wallspace $(\mathbf{R}, \mathcal{W}_\reals, \mathcal{B}, \mu)$ of Example~\ref{exmp:the Line}, with its natural $\integers$-action is $\integers$-standard with fundamental domain $\omega(0, 1)$.
\end{exmp}

\begin{exmp}\label{exmp:cont cube axis standard}
Let $\integers$ act freely on a CAT(0) cube complex $ X$,
and let $(X, \mathcal{W}_X, \mathcal{B}_X, \mu_X)$ be the associated continuous measured wallspace of Example~\ref{exmp:CAT0CubeComplexes}.
Suppose $\integers$ stabilizes a combinatorial axis $A$.
Consider the measure space whose elements are the walls intersecting $A$,
and whose measurable subsets, and measure is induced from the measured wallspace structure on $\tilde X$.
This will be isomorphic mod-zero to $(\mathbf{R}, \mathcal{W}_\reals, \mathcal{B}_\reals, \lambda)$.
Then $(X, \mathcal{W}_X, \mathcal{B}_X, \mu_X)$ is $\integers$-standard.
\end{exmp}

\begin{exmp}\label{exmp:hyperbolic axis standard}
Recall the measured wallspace structure on $\hyperbolic^n$ in
%\begin{com} I do not understand why we cannot just cite Examples~\ref{exmp:hyperbolicSpace} to see that for $a\neq b$, the subspace $\omega(a,b)$ is standard (after scaling). You already claimed there that the Borel sets are measurable. The standard criterion kicks in by realizing that this subset is homeomorphic to $[a,b]\times B^{n-1}$. So the key point is to prove that. One imagines that it is true by using a closed ball $B^{n-1}$, and then using a bijective continuous map from a compact set to a hausdorff space. And then removing the measure zero (why?) set of hyperplanes containing $\gamma$. I guess you will need you to explain to me.\end{com}
 Example~\ref{exmp:hyperbolicSpace}.
 In Lemma~\ref{lem:hyperbolicDispersal}, we identified the subset of walls transversely intersecting a finite geodesic in $\hyperbolic^n$ with the points in $[a,b] \times B^{n-1}$, where $B^{n-1}$ is the open ball of dimension $n$, but we did not determine the topology on the set of such walls.
 We will now show that the subspace of walls intersecting a bi-infinite geodesic is homeomorphic to $\reals \times B^{n-1}$.
Let $\gamma \subseteq \hyperbolic^n$ be a bi-infinite geodesic, and let $\Lambda$ be a hyperplane such that $\gamma \cap \Lambda = \{ p \}$.
There is a subspace $R \subset \Isom(\hyperbolic^n)$, consisting of rotations at $p$, with $R$ homeomorphic to the open ball $B^{n-1}$ such that $R\Lambda$ is the set of all hyperplanes intersecting $\gamma$ transversely at $p$.
Let $U_\gamma \subseteq \mathcal{H}$ be the subspace consisting of hyperplanes intersecting $\gamma$ transversely.
Each hyperplane in $U_\gamma$ is obtained by rotating $\Lambda$ by an element of $R$ and then translating along $\gamma$, and conversely each such transformation gives a unique hyperplane in $U_\gamma$.
Therefore, there is a subspace $V_\gamma \subseteq \Isom(\hyperbolic^n)$ that is homeomorphic to $\reals \times B^{n-1}$, that injects in $\Isom(\hyperbolic^n) / \stabilizer(\Lambda)$ such that $V_\gamma \Lambda = U_\gamma$.
As in Example~\ref{lem:hyperbolicDispersal}, the subset $U_\gamma \subseteq \mathcal{H}$ can be identified with the set $\reals \times B^{n-1}$, and is the image of $V_\gamma$ in $\mathcal{H} = \Isom(\hyperbolic^n) / \stabilizer(\Lambda)$.

Note that $\Isom(\hyperbolic^n) / \stabilizer(\Lambda)$ is Hausdorff since $\stabilizer(\Lambda)$ is a closed subgroup of a Lie Group. \begin{com} Bourbaki General Topology III 2.5 Prop 13 \end{com}
Thus $U_\gamma$ is Hausdorff.
Hence, the restriction $V_\gamma \rightarrow U_\gamma$ is a topological embedding on each compact set.
Since $V_\gamma$ is locally compact, we see that $V_\gamma \rightarrow U_\gamma$ is also an open map, and we conclude it is a homeomorphism.
We conclude that $U_\gamma$ is homeomorphic to $\reals \times B^{n-1}$.

Observe that there are no atoms in this measure space.
 Let $Z\subset \Isom(\hyperbolic^n)$ be a cyclic group acting freely on $\gamma$ and hence on $U_\gamma$.
  Then $Z$ preserves the product structure and acts freely on the $\reals$-factor with some fundamental domain $(a,b)$, and hence the action of $Z$  on $U_\gamma$ has fundamental domain $(a,b)\times B^{n-1}$.
  As observed in Example~\ref{exmp:hyperbolicSpace}, the Borel sets are measurable.
   Since $(a,b) \times B^{n-1}$ is metrizable, complete, and separable it is a standard probability space
by Theorem~\ref{thm:topological characterization}.
Consequently, $(U_\gamma, \mathcal{B}_{\hyperbolic} \cap U_\gamma , \mu)$ is $\integers$-standard.
\end{exmp}

%Another example:
%Let $\integers$ act semi-simply\begin{com}needs definition\end{com} on a measured wallspace.
%Then the induced measured wallspace on $\integers \tilde x$ is $\integers$-standard.
%
%A nonexample: the set of walls that partition $\integers$ in any $\integers\subset G$ with $G$ amenable
%equipped with the measured wallspace of Subsection~\ref{sub:amenable}.
%\begin{com} stuff to fight about here\end{com}

\section{Main Theorem}
\label{sec:main}
%\begin{defn}[Cut, Skim, Disjoint]
%Let $(X, \mathcal{W}, \mathcal{B}, \mu)$ be a measured wallspace with a $Z$-action, where $Z = \langle z\rangle$ is infinite cyclic.
%Let $x \in X$ and $\Lambda \in \mathcal{W}$.
%The wall $\Lambda$ \emph{cuts} $Zx$ if $\{z^nx\}_{n\geq m}$ and $\{z^nx\}_{n< m}$ are separated by $\Lambda$ for some $m$.
%The wall $\Lambda$ \emph{skims} $Zx$ if $\Lambda$ separates some pair of elements in $Zx$, but does not cut $Zx$.
%The wall $\Lambda$ is \emph{disjoint} from $Zx$ if $\Lambda$ doesn't cut or skim $Zx$.
%Note that since $Z$ is countable, the sets of walls that separate, skim, and are disjoint from $Z x$ belong to $\mathcal{B}$.
%\end{defn}

\subsection{Graphs of Groups with Measured Wallspaces}

\begin{defn}[Cut, Skim, Disjoint]
Let $(X, \mathcal{W}, \mathcal{B}, \mu)$ be a measured wallspace with a $G$-action, for some group $G$.
Let $x \in X$ and $\Lambda \in \mathcal{W}$.
The wall $\Lambda$ \emph{cuts} $Gx$ if  each halfspace of $\Lambda$ contains infinitely many elements of $Gx$.
The wall $\Lambda$ \emph{skims} $Gx$ if one halfspace contains a finite, non-empty subset of $Gx$.
The wall $\Lambda$ is \emph{disjoint} from $Gx$ if one halfspace is empty.
Note that if $G$ is countable, the sets of walls that separate, skim, and are disjoint from $Gx$ belong to $\mathcal{B}$.
\end{defn}

For $(X,\mathcal W, \mathcal{B}, \mu)$ and $\lambda >0$, we define the \emph{scaled measured wallspace}
$(X ,\mathcal W , \mathcal{B}, \lambda \mu)$  with property that its measure has been scaled by $\lambda$.

\begin{defn}
A splitting of a group $G$ as a \emph{graph of groups with measured wallspaces} consists of a simplicial graph of groups $\Gamma$ such that:
 \begin{enumerate}
  \item Each vertex group $G_v$ acts metrically properly on a measured wallspace $(X_v,\mathcal W_v, \mathcal{B}_v, \mu_v)$.
  \item Each edge group $G_e$ acts metrically properly on a measured wallspace $(X_e,\mathcal W_e, \mathcal{B}_e, \mu_e)$.
  \item For each edge $e$ incident to $v$, there is a chosen $x_v^e \in X_v$.
       Let $\mathcal{W}_v^e \subseteq \mathcal{W}_v$ denote the set of walls that cut $G_ex_v^e$.
      Let $\mathcal{B}_v^e \subset \mathcal{B}_v$ be the $\sigma$-algebra consisting of the subsets of $\mathcal{W}_v^e$.
 \item\label{ax:skim measure zero} The set of walls in $\mathcal W_v$ that skim $G_ex_v^e$ have measure zero.
  \item\label{ax:edge inclusion} For each edge $e$ adjacent to $v$, let $\varphi^e_{v}: G_e \rightarrow G_{v}$ denote the corresponding inclusion. There is $\varrho^e_{v}>0$ and a $G_e$-equivariant mod-zero isomorphism $ \phi_v^e: (\mathcal W_e, \mathcal{B}_e, \varrho_v^e \mu_e) \rightarrow (\mathcal{W}_v^e, \mathcal{B}_v^e, \mu_{v})$.
 \end{enumerate}
\end{defn}

Note that the requirement that $\Gamma$ be simplicial is purely for notational convenience; if both endpoints of an edge were at the same vertex it would be inconvenient to specify the attaching maps. In any case, by possibly passing to an index~2 subgroup, one can always ensure simpliciality while maintaining the hypotheses of the following result which is our main theorem.

\begin{thm}\label{thm:main}
Let $G$ split as a graph $\Gamma$ of groups with measured wallspaces and:
\begin{enumerate}
%\item \label{ax:invariance} when we make this into a group acting on a tree, we will say that the measure is preserved or scaled...
\item Each edge group $G_e = \langle g_e \rangle$ is infinite cyclic.
% \item \label{ax:halvesAxis} Let $e$ be an edge incident to $v$, if $\Lambda$ cuts $G_ex_v^e$, then $\Lambda$ partitions $G_ex_v^e$ into the sets $\{g_e^nx_v^e \}_{m > M_v^e}$ and $\{g_e^n x_v^e \}_{m \leq M_v^e}$ for some $M_v^e \in \integers$.
%\item \label{ax:Zstandard} For each $e$ there is $\varrho^e_v >0$ and an equivariant isomorphism $(X_e, \mathcal{W}_e, \mathcal{B}_e,  \mu_e) \rightarrow (\mathbf{R}, \mathcal{W}_\reals, \mathcal{B}_\reals, \lambda)$, where $g_e$ maps to $1 \in \integers$, and $x_e$ maps to $0 \in \mathbf{R}$.
\item \label{ax:funDomain} $(\mathcal{W}_e, \mathcal{B}_e, \mu_e)$ has fundamental domain $\omega(x_e, g_ex_e)$.
%\item \label{ax:funDDomain} For each edge $e$ incident to a vertex $v$, the inclusion     $(\mathcal W_e, \mathcal{B}_e, \varrho_v^e\mu_e) \hookrightarrow (\mathcal{W}_v^e, \mathcal{B}_v^e, \mu_{v})$ is a $\integers$-equivariant isomorphism.
\item \label{ax:funDDDomain} The image $\phi_v^e(\omega(x_e, g_ex_e)) = \omega ( x_v^e, \varphi_v^e(g_e) x_v^e)$ for all edges $e$ incident to a vertex $v$.
\item \label{ax:countableVertices} The Bass-Serre tree $T$ has countably many vertices.
 %\item \label{ax:metprop on vertices} Let $x_e,x_e'$ denote the chosen points in the initial and terminal vertex spaces $X_v, X_{v'}$ of $e$. We require that the set of walls \emph{separating} $x_e,x_e'$ is measurable and of finite measure. This consists of the walls in $(S_e,\mu_v)$   whose images in $(S_{e'},\mu_{v'})$ have the property that $x_e, x_{e'}$ lie in non-corresponding halves under the bijection. ????? \begin{com}Since we are only dealing with walls that cut the orbit in a suitable sense, The bijection between the halfspaces of a wall is determined because of the action.\end{com} \begin{com} More organized to choose a basepoint in $X_e$ and then make the above finite measure condition applicable to the map $\mathcal W_e\rightarrow \mathcal S_e$. \end{com}
\item\label{ax:dispersed} The subgroup $G_e \leqslant G_v$  is dispersed relative to $ x_v^e \in X_v$.
%\item \label{ax:altBPseparation} Let $R(v, e) = \Big\{ \Lambda \in \mathcal{W}_e \mid \phi_v^e(\Lambda) \in \bigcup_{m\in \{0,1,2,\ldots \}} \omega(\varphi_v^e(g_e^m)x_v^e, \varphi_v^e(g_e^m)x_v^e) \Big\}$ and  $\; L(v,e) = \Big\{ \Lambda \in \mathcal{W}_e \mid \phi_v^e(\Lambda) \in \bigcup_{m\in \{0,1,2,\ldots \}} \omega(\varphi_v^e(g_e^{-m})x_v^e, \varphi_v^e(g_e^{-m})x_v^e) \Big\}$.      Then $\mu_e \big(L({v_1},e) \cap R({v_2}, e) \big) < \infty$ and $\mu_e \big(L({v_2},e) \cap R({v_1}, e) \big) < \infty$. \begin{com}This generalizes Condition~\eqref{ax:funDDDomain}.\end{com}
%
% Let $Q^{\pm}(x_{\tilde v}^{\tilde e}) =
%    \{\Lambda \in \mathcal{W}_e \mid \overleftarrow{\phi_{v_1}^e(\Lambda)} \ni \varphi_{v_1}^e(g_e^{\pm m}) x_{v_1}^e \; \textrm{for all $m$ sufficiently large, and} \; x^e_{v_i} \in \overleftarrow{\phi_{v_1}^e(\Lambda)}  \} $
%    and $P^{\pm}(x_{\tilde v}^{\tilde e}) =
%    \{\Lambda \in \mathcal{W}_e \mid \overleftarrow{\phi_{v_1}^e(\Lambda)} \ni \varphi_{v_1}^e(g_e^{\pm m}) x_{v_1}^e \; \textrm{for all $m$ sufficiently large, and} \;  x^e_{v_i} \notin \overleftarrow{\phi_{v_1}^e(\Lambda)} \} $.
%   Then $\mu_e(Q^{\pm}(x_{v_1}^e) \cap P^{\pm}(x_{v_2}^e)) < \infty$.
\end{enumerate}
\noindent Then $G$ is aTmenable.
\end{thm}

\begin{rem}
 Condition~\eqref{ax:funDomain} is equivalent to following condition: almost all $\Lambda$ partition $G_ex_v^e$ as $\{g_e^nx_v^e \}_{m > M} \; \sqcup \; \{g_e^n x_v^e \}_{m \leq M}$ for some $M \in \integers$ depending on $\Lambda, v,e$.
\end{rem}

\begin{rem}
 Condition~\eqref{ax:funDDDomain} can be replaced by the following more general condition which is not required for our applications: For any edge $e\in \Gamma$ with endpoints $v_1,v_2$, we require that
 $\mu_e \big(L({v_1},e) \cap R({v_2}, e) \big) < \infty$ \begin{com} only need one of these since we do not direct $e$\end{com} and $\mu_e \big(L({v_2},e) \cap R({v_1}, e) \big) < \infty$ where:
  $$R(v, e) = \Big\{ \Lambda \in \mathcal{W}_e \mid \phi_v^e(\Lambda) \in \bigcup_{m\in \{0,1,2,\ldots \}} \omega(\varphi_v^e(g_e^m)x_v^e, \varphi_v^e(g_e^{m+1})x_v^e) \Big\}$$
 $$ L(v,e) = \Big\{ \Lambda \in \mathcal{W}_e \mid \phi_v^e(\Lambda) \in \bigcup_{m\in \{0,1,2,\ldots \}} \omega(\varphi_v^e(g_e^{-m})x_v^e, \varphi_v^e(g_e^{-(m+1)})x_v^e) \Big\}$$

 \noindent This is what would ensure that $\omega(x_{\tilde{v}_{i-1}}^{\tilde{e}_i},x_{\tilde{v}_i}^{\tilde{e}_{i}})$ is a finite measurable set in Lemma~\ref{lem:separatingWallsAreMeasurableSet}.
 \end{rem}

\begin{defn}[Monic] \label{defn:monic}
If a group $G$ satisfies the hypotheses of Theorem~\ref{thm:main}, and has the additional property that each $\varrho_v^e = 1$, then we say that the splitting of $G$ is \emph{monic}.
Much of our discussion will focus on the monic special case.
\end{defn}

Before outlining the proof of Theorem~\ref{thm:main}, we show that it implies Theorem~\ref{thm:typical application}.

\begin{proof}[Proof of Theorem~\ref{thm:typical application}]
 Suppose that $G$ splits as a graph of groups $\Gamma$.
 We will show that $G$ satisfies the criterion of Theorem~\ref{thm:main}. \begin{com}
 	CORRECTION: Inserted explicit statement at the start of the proof that we will be verifying the criterion of Theorem~\ref{thm:typical application}.
 \end{com}
 Each edge group $G_e \cong \integers$ acts on $(\mathbf{R}, \mathcal{W}_\reals, \mathcal{B}_\reals, \mu)$ with fundamental domain $\omega(0,1)$.

 If $G_v$ acts properly and semisimply on $\hyperbolic^n$,%is isomorphic to a Kleinien group,
 then we let $(\hyperbolic^n, \mathcal{H}, \mathcal{B}_\hyperbolic, \mu)$ from Example~\ref{exmp:hyperbolicSpace} be the associated wallspace.
 If the edge $e$ in $\Gamma$ is incident to $v$, we let $x_v^e \in \hyperbolic^n$ be a point in the geodesic axis of $G_e$.

 Note that every wall in $\mathcal{H}$ is either disjoint from, or intersects in precisely one point the geodesic containing $x_v^e$, hence Condition~\eqref{ax:funDomain} holds in this case.
 By Lemma~\ref{lem:hyperbolicDispersal}, the $G_e$ is dispersed relative to $x_v^e$ in their geodesic axis.
 If $G_v$ is virtually special, then by Proposition~\ref{prop:DispersedRAAGS} \begin{com}
 	I corrected the result cited after rewriting.
 \end{com}
 %and Proposition~\ref{prop:basepointInvarianceCAT0}
 there is a CAT(0) cube complex $X_v$ such that each edge group $G_e$ is dispersed relative to a basepoint $x_v^e$ in a combinatorial axis for $G_e$.
  Therefore, Condition~\eqref{ax:dispersed} holds in the respective cases.
 And we let $(X_v, \mathcal{W}_v, \mathcal{B}_v, \mu_v)$ be the continuous wallspace of Example~\ref{exmp:CAT0CubeComplexes}.
 Combinatorial axis are intersected by hyperplanes at most once.
 Hence, as $x_v^e$ is contained in a combinatorial geodesic axis Condition~\eqref{ax:funDomain} holds.

 As described in Examples~\ref{exmp:cont cube axis standard}~and~\ref{exmp:hyperbolic axis standard} the walls intersecting a geodesic axis stabilized by the edge group are $\integers$-standard, and have fundamental domains given by the walls separating two consecutive points in an orbit.
 Hence, for each edge $e$ incident to a vertex $v$, there is a $G_e$-equivariant isomorphism $\phi_v^e$ that satisfies Condition~\eqref{ax:funDomain}. Moreover, we can choose $\phi_v^e$ such that the image $\phi_v^e(\omega(x_e, g_ex_e)) = \omega ( x_v^e, \varphi_v^e(g_e) x_v^e)$, therefore satisfying Condition~\eqref{ax:funDDDomain}.

 Thus $G$ is aTmenable by Theorem~\ref{thm:main}.
\end{proof}

\subsection{Outline of Proof of Theorem~\ref{thm:main}}

The proof of Theorem~\ref{thm:main} is broken up into the following steps:
We describe a homomorphism $G \rightarrow \reals^*$ in Lemma~\ref{lem:modularHomomorphism}.
By Lemma~\ref{lem:aTmen by amen} and the fact that subgroups of $\reals^*$ are amenable, it suffices to show that its kernel $G'$ is aTmenable.
By design, $G'$ has the attractive property that its splitting is monic.
We prove the monic special case of Theorem~\ref{thm:main} by first constructing a measured wallspace $(X,\mathcal{W},\mathcal{B},\mu)$ that $G$ acts on.
Then we show that $G$ acts metrically properly on $(X,\mathcal{W},\mathcal{B},\mu)$ in Proposition~\ref{prop:actionIsMetricallyProper}.
The aTmenability of $G$ then follows from Lemma~\ref{lem:prop measured wallspace action gives atmenable}.

\subsection{The modular homomorphism} \label{sub:modular}

Our initial goal is the following result:

\begin{lem} \label{lem:modularHomomorphism} Let $G$ split as a graph of groups with measured wallspaces satisfying the hypotheses of Theorem~\ref{thm:main}.
There is a homomorphism $G\rightarrow \reals^*$ whose kernel $G'$ is monic.
\end{lem}

\begin{proof}
We orient the edges of $\Gamma$ to obtain a directed graph.
For an edge $e$ from $u$ to $v$ we use the notation $\varrho_e^+ = \varrho_v^e$ and $\varrho_e^- = \varrho_u^e$.
To each directed edge we associate the nonzero real weight $w(e) = \frac{\varrho_e^{+}}{\varrho_e^{-}}$, and let $w(e^{-1}) = w(e)^{-1}$.
The \emph{modular homomorphism} $ f : \pi_1 \Gamma \rightarrow \reals^*$ is induced by the function that maps a closed combinatorial path $e_{i_1} \cdots e_{i_n}$ to $w(e_{i_1}) \cdots w(e_{i_n})$.
Let $\widehat{\Gamma}$ denote the covering space of $\Gamma$ associated to the kernel of $f$.
Let $G' \leqslant G$ be the subgroup associated to $\widehat{\Gamma}$.

 Assign a copy of the measured wallspace $(X_v, \mathcal{W}_v, \mathcal{B}_v, \mu_v)$ to $\widehat{v}$, where $\widehat{v}$ lies in the fiber of $v \in \Gamma$, and similarly for each edge space and edge space inclusion.
For each edge $\widehat{e}$ incident to a vertex $\widehat{v}$ we let $x_{\widehat{e}}$ be the point corresponding to $x_e \in X_v$, where $\widehat{e}$ lies in the fiber of $e$ and $\widehat{v}$ in the fiber of $v$.
After choosing a basepoint in $\widehat{\Gamma}$ we can scale the measured wallspaces at each vertex so that $\varrho_{\widehat{e}}^{\pm} = 1$ for all $\widehat{e}$ in $\widehat{\Gamma}$.
Each edge group remains infinite cyclic and dispersed in its respective vertex groups.
\end{proof}

\subsection{Constructing the Measured Wallspace}

Assuming that the splitting of $G$ is monic, we construct a measured wallspace $(X, \mathcal{W}, \mathcal{B}, \mu)$ that $G$ acts on.
Let $T$ be the Bass-Serre tree of the splitting of $G$.
For each $\tilde{v} \in T$, we let $(X_{\tilde{v}}, \mathcal{W}_{\tilde{v}}, \mathcal{B}_{\tilde{v}}, \mu_{\tilde{v}})$ be a copy of $(X_v, \mathcal{W}_{v}, \mathcal{B}_{v}, \mu_{v})$ with the associated
$G_{\tilde v}$ action.
Let $X = \bigsqcup_{\tilde{v} \in V(T)} X_{\tilde v}$ be the disjoint union of these \emph{vertex spaces}.
Note that $G$ acts on $X$, and thus on $\bigsqcup_{\tilde{v} \in V(T)} \mathcal{W}_{\tilde{v}}$, and $\bigsqcup_{\tilde{v} \in V(T)} \mathcal{B}_{\tilde{v}}$ such that $\mu_{\tilde{v}}(U) = \mu_{g\tilde{v}}(gU)$ for $U \in \mathcal{B}_{\tilde{v}}$ and $g \in G$.

 For each edge $e$ of $\Gamma$ \begin{com} Added ``of $\Gamma$''\end{com} with endpoints $u, v$ choose a lift $\tilde e_o$ and let $\tilde u_o, \tilde v_o$ denote its endpoints in $T$.
 Our identification of $X_{\tilde u_o}$ with $X_u$ allows us to choose a point $\tilde x_{\tilde u_o}^{\tilde e_o} \in X_{\tilde u_o}$ corresponding to $x_u^e$.
 For each coset $gG_e$ we fix the representative $g \in G$ and let $x_{g \tilde u_o}^{g \tilde e_o}$ equal $g x_{\tilde u_o}^{\tilde e_o}$.
 We likewise define each $ x_{g\tilde v_o}^{g \tilde e_o}$.
 Having made these choices, for an edge $\tilde e$ incident to a vertex $\tilde v$, Condition~\eqref{ax:dispersed} ensures the orbits $\{ g G_{\tilde{e}} x_{\tilde{v}}^{\tilde{e}} \}_{g \in G_{\tilde v}}$ are dispersed in $(X_{\tilde{v}}, \mathcal{W}_{\tilde{v}}, \mathcal{B}_{\tilde{v}}, \mu_{\tilde{v}})$.

In a similar vein, for each edge $\tilde e$ in $T$ projecting to an edge $e$ in $\Gamma$ we let $(X_{\tilde e}, \mathcal{W}_{\tilde e}, \mathcal{B}_{\tilde e}, \mu_{\tilde e})$ be a copy of $(X_e, \mathcal{W}_e, \mathcal{B}_e, \mu_e)$.
If $\tilde e$ is incident to $\tilde v$, then we let $\phi_{\tilde{v}}^{\tilde{e}}$ be a copy of $\phi_v^e$ mapping $(\mathcal{W}_{\tilde e}, \mathcal{B}_{\tilde e}, \mu_{\tilde e})$ isomorphically to $(\mathcal{W}_{\tilde v}^{\tilde e}, \mathcal{B}_{\tilde v}^{\tilde e}, \mu_{\tilde v})$.
The action of $G$ satisfies $g\phi_{\tilde v}^{\tilde e} = \phi_{g \tilde v}^{g \tilde e}$.

Each edge $\tilde e$ of $T$ determines a \emph{vertical wall} whose halfspaces are the sets of vertex spaces corresponding to the two sets of vertices separated by $\tilde e$.

We shall now define an equivalence relation $\sim$ on $\bigsqcup_{\tilde v \in V(\Gamma)} \mathcal{W}_{\tilde v}$.
Let $\tilde{e}$ be an edge in $T$ joining $\tilde u$ to $\tilde v$, and let $\Lambda_1 \in \mathcal{W}_{\tilde u}$ and $\Lambda_2 \in \mathcal{W}_{\tilde{v}}$, then $\Lambda_1 \sim \Lambda_2$  if there is $\Lambda \in \mathcal{W}_{\tilde{e}}$ such that $\phi^{\tilde{e}}_{\tilde{u}}(\Lambda) = \Lambda_1$ and $\phi^{\tilde{e}}_{\tilde{v}}(\Lambda) = \Lambda_2$.
More specifically there is a correspondence between the halfspaces of $\Lambda_1$ and the halfspaces of $\Lambda_2$, determined by declaring two halfspace to be in correspondence precisely when they both contain all sufficiently large positive or negative translates $\varphi_{\tilde v}^{\tilde e}(g_{\tilde{e}}^m) x_{\tilde v}^{\tilde e}$ and $\varphi_{\tilde v}^{\tilde e}(g_{\tilde{e}}^m) x_{\tilde u}^{\tilde e}$.
Any wall that belongs to the equivalence class of a wall in the null set excluded from the isomorphisms $\phi^{\tilde e}_{\tilde{u}}, \phi^{\tilde e}_{\tilde{v}}$ is discarded, and referred to as a \emph{discarded wall}.
We also discard any walls that skim an orbit $G_{\tilde e}x_{\tilde v}^{\tilde e}$.
Condition~\eqref{ax:countableVertices} ensures that
 $V(T)$ is countable and so we deduce that the set of discarded walls in any $\mathcal{W}_{\tilde{v}}$ is a nullset.
As $T$ is a tree, each equivalence class contains at most one wall in each $\mathcal{W}_{\tilde{v}}$.
 Let $T_{\Lambda}$ be the subtree of $T$ spanned by the vertices $\tilde{v}$ such that $X_{\tilde{v}}$ contains a wall in the equivalence class of $\Lambda$.

The \emph{horizontal walls} are constructed from these equivalence classes.
If $\Lambda$ is a wall in $\mathcal{W}_{\tilde{v}}$ that has not been discarded, then the wall corresponding to its equivalence class partitions $X_{\tilde{v}}$ by taking the union of the corresponding equivalence classes of halfspaces for all $\tilde v \in V(T_\Lambda)$.
If $\tilde{v} \notin V(T_\Lambda)$, then let $\tilde{u}$ be the closest vertex in $T_{\Lambda}$ to $\tilde{v}$, and $\tilde e$ the edge separating them.
Let $\Lambda' \in \mathcal{W}_{\tilde{u}}$ such that $\Lambda' \sim \Lambda$.
Then  $G_{\tilde{e}} x_{\tilde{v}}^{\tilde{e}} \subset X_{\tilde{u}}$ is disjoint from $\Lambda'$ and hence is contained in either the left or right halfspace of $\Lambda' \in \mathcal{W}_{\tilde{u}}$.
Accordingly, $X_{\tilde{v}}$ is added to the left or right halfspace.

Let $\mathcal W^\hh$ and $\mathcal W^\vv$ denote the horizontal and vertical walls.
Let $\mathcal W=\mathcal W^\hh \sqcup \mathcal W^\vv$ denote the set of all walls.
There is a natural map, modulo the nullset of discarded walls, $\mathcal{W}_{\tilde{v}} \hookrightarrow \mathcal{W}^{\hh}$ that takes a wall to the horizontal wall constructed from the equivalence class containing it.

Define the measurable subsets $\mathcal{B}^{\hh}$ of $\mathcal W^\hh$ to be the $\sigma$-algebra generated by
the inclusion of the elements of $\mathcal{B}_{\tilde{v}}$ into $\mathcal{W}^{\hh}$.
Define the measurable subsets $\mathcal{B}^{\vv}$ of $\mathcal W^\vv$ to be the collection of all subsets of vertical walls.
Let $\mathcal W^\hh_{\tilde v}$ denote the image of the embedding of $\mathcal W_{\tilde v}$ in $\mathcal{W}^{\hh}$.
Let $\mathcal{B}$ be the $\sigma$-algebra generated by $\mathcal{B}^{\vv} \cup \mathcal{B}^{\hh}$.

Let $\{ \tilde{e}_1, \tilde{e}_2 \ldots \}$ be an enumeration of representatives of the edge orbits.
If $\Lambda_{g\tilde{e}_i}$ is the vertical wall corresponding to $g\tilde{e}_i$ then let $\mu^{\vv}(\Lambda_{g\tilde{e}_i}) = i$.

%\begin{com} BAH : Presumably, there is a textbook theorem that says that given measures on subcollection of walls, there is an induced measure on the sigma algebra they generate that agrees with each individually, provided that they agree (pairwise?) on intersections.\end{com}

The measure on $\mathcal{B}^\hh$ is defined as follows:
%\begin{com}We assume that $T$ has countably many vertices.\end{com}
Choose an enumeration $\{\tilde{v}_1, \tilde{v}_2, \ldots \}$ of the vertices of $T$.
Given a measurable subset $U \in \mathcal{B}^\hh$, we partition it as
$$U = \bigsqcup_{i=1}^\infty U_i,$$
where we recursively define
%\begin{equation}\label{eq:decompositing U}
 \[ U_i = \mathcal W^\hh_{v_i} \cap \big(U - \cup_{j=1}^{i-1} U_j \big).
 \]
 %\end{equation}
\noindent Note each $U_i$ is measurable because each $\mathcal W^\hh_{v_i}$ is measurable.

We define the measure $\mu^\hh$ as follows:
\begin{equation}\label{eq:measuring U}
\mu^\hh(U)=\sum_{i=1}^\infty \mu_i(U_i)
\end{equation}
Observe that $\mu^\hh$ is a measure because firstly $\mu^\hh(\emptyset)=0$, and secondly if $\{A_i\}_{i\in \naturals}$ is a sequence of disjoint measurable sets,
then $\mu^\hh(A)=\sum_{i=1}^\infty \mu^\hh(A_i)$ by rearranging the sum in Equation~\eqref{eq:measuring U}.

Finally, let $\mu = \mu^\hh + \mu^\vv$.

\begin{lem}\label{lem:key point}
Let $A$ be a measurable set of walls such that $A\subset \mathcal W^\hh_{\tilde{u}}$ and $A\subset \mathcal W^\hh_{\tilde{v}}$
for vertices $\tilde{u},\tilde{v}$ of $T$.
Then $\mu_{\tilde{u}}(A) = \mu_{\tilde{v}}(A)$.
\end{lem}
\begin{proof}
Let $\tilde{u}=\tilde{v}_0, \tilde{v}_1,\ldots, \tilde{v}_n=\tilde{v}$ be the sequence of vertices on a geodesic path $\tilde{e}_1\tilde{e}_2\cdots \tilde{e}_n$ in $T$ from $\tilde{u}$ to $\tilde{v}$.
Observe that $A\subset \mathcal W^\hh_{\tilde{v}_i}$ for each $i$.
Then note that $\mu_{\tilde{v}_i}(A)=\mu_{\tilde{v}_{i+1}}(A)$ since the splitting of $G$ is monic and hence $\varrho_{\tilde{e}_i}^{\pm} = 1$.
\end{proof}

\begin{lem}\label{lem:welldefined}
The measure $\mu^\hh$ is well-defined, in the sense that it does not depend on the enumeration of the vertices.
\end{lem}
\begin{proof}
Suppose $U=U_1\sqcup U_2 \sqcup U_3 \cdots$ is a decomposition of a measurable set $U$ with respect to an enumeration $\{ \tilde u_1, \tilde u_2, \ldots \}$,
and
$U=V_1\sqcup V_2 \sqcup V_3 \cdots$ is a decomposition with respect to $\{ \tilde v_1, \tilde v_2, \ldots \}$.
(For brevity we use the notation $\mu_i = \mu_{\tilde u_i}$ and $\mu_j = \mu_{\tilde v_j}$.)
The claim that $\sum \mu_i(U_i)=\sum \mu_j(V_j)$ follows from the following equation where the second equality holds by Lemma~\ref{lem:key point}.
$$\sum_i \mu_i(U_i) =  \sum_i\sum_j \mu_i(U_i\cap V_j) = \sum_i\sum_j \mu_j(U_i\cap V_j) = \sum_j\sum_i \mu_j(U_i \cap V_j) = \sum_j \mu_j(V_j)  \ \ \qedhere$$
\end{proof}

\begin{lem}
The measure $\mu^\hh$ is $G$-invariant.
\end{lem}
\begin{proof}
This holds from the following equation. \begin{com}CORRECTION: THis used to say ``The holds''. \end{com}
Its first equality follows from the definition with respect to an enumeration $\{\tilde u_1, \tilde u_2, \ldots\}$.
Its second equality follows from the fact that $\mu_{\tilde u_i}(U_i) = \mu_{g \tilde u_i}(gU_i)$.
Its final equality follows by defining a second enumeration $\tilde v_i = g \tilde u_i$ and applying Lemma~\ref{lem:welldefined}.
$$\mu^\hh(U) = \sum_i \mu_{\tilde u_i}(U_i) = \sum_i \mu_{g \tilde u_i}(gU_i) = \mu^\hh(gU) \hfill \qedhere$$
\end{proof}

%
%\begin{rem}
%If T were locally finite we could have defined $\mu(U)=\limsup_F \sum_F \mu(U_i\cap U)$ where the limsup is taken over all
%collections of disjoints sets $\{U_i\}$ where each $U_i$ is a measurable subset of some $\mathcal W_v$.
%
%Then we don't need to assume $T$ is countable.
%And the $G$-invariance now becomes immediate.
%However the claim that it is a measure will now require Lemma~\ref{lem:key point}.
%\end{rem}
%

\begin{lem} \label{lem:separatingWallsAreMeasurableSet}
 Let $a, b \in X$, then $\omega( a,  b) \in \mathcal{B}$ and $\mu(\omega(a,b)) < \infty$.
\end{lem}

\begin{proof}
  If ${a}$ and ${b}$ are in the same vertex space then $\omega({a}, {b}) \in \mathcal{B}$ by construction.
  Otherwise, consider the geodesic in $T$ between the vertices $\tilde{v}_0, \tilde{v}_n$ that ${a}, {b}$ project to.
  Let $\tilde{v}_0, \tilde{v}_1 ,\ldots, \tilde{v}_n$ be the corresponding sequence of vertices in $T$, and let $\tilde{e}_1, \tilde{e}_2,\ldots, \tilde{e}_n$ be the connecting edges, and assume for notational purposes that each $\tilde{e}_i$ is directed from $\tilde{v}_{i-1}$ to $\tilde{v}_i$.
  We thus have the following sequence of points
    where $x_{\tilde{v}_{i-1}}^{\tilde{e}_i} \in X_{\tilde{v}_{i-1}}$ and $x_{\tilde{v}_i}^{\tilde{e}_i} \in X_{\tilde{v}_{i}}$.
  %For each $i$, we let $x_i$ and $y_i$ be the basepoints in $X_{\tilde{v}_i}, X_{\tilde{v}_{i+1}}$ corresponding to $\tilde{e}_i$.
  $$a, x_{\tilde{v}_0}^{\tilde{e}_1}, x_{\tilde{v}_1}^{\tilde{e}_1}, x_{\tilde{v}_1}^{\tilde{e}_2}, x_{\tilde{v}_2}^{\tilde{e}_2}, \ldots , x_{\tilde{v}_{n-1}}^{\tilde{e}_n}, x_{\tilde{v}_n}^{\tilde{e}_n}, b$$
  The set $\omega(a,b)$ consists of walls in $\mathcal{W}$ separating an odd number of consecutive elements of the sequence.
   By construction, $\omega(a,x_{\tilde{v}_0}^{\tilde{e}_1})\in \mathcal{B}$, and
        $\omega(x_{\tilde{v}_i}^{\tilde{e}_i} , x_{\tilde{v}_i}^{\tilde{e}_{i+1}})\in \mathcal{B}$, and $\omega(x_{\tilde{v}_n}^{\tilde{e}_n}, b) \in \mathcal{B}$.
  We claim that $\omega(x_{\tilde{v}_{i-1}}^{\tilde{e}_i},x_{\tilde{v}_i}^{\tilde{e}_{i}})$ consists of a single vertical wall.
   Indeed suppose $\omega(x_{\tilde{v}_{i-1}}^{\tilde{e}_i},x_{\tilde{v}_i}^{\tilde{e}_{i}})$ contained a horizontal wall $\Lambda$.
   Then $\Lambda$ would have been constructed an equivalence class containing walls $\Lambda_{i-1} \in \mathcal{W}_{\tilde v_{i-1}}$ and $\Lambda_i \in \mathcal{W}_{\tilde v_i}$ such that $x_{\tilde v_{i-1}}^{\tilde e_i}$ is in the left halfspace of $\Lambda_{i-1}$ but $x_{\tilde v_i}^{\tilde e_i}$ is in the right halfspace of $\Lambda_i$ (or vice versa).
   Thus $\Lambda_{i-1} \in \varphi_{\tilde v_{i-1}}^{\tilde e_i}(g^{p}_{\tilde e_i}) \omega( x_{\tilde v_{i-1}}^{\tilde e_i}, \varphi_{\tilde v_{i-1}}^{\tilde e_i}(g_{\tilde e_i}) x_{\tilde v_{i-1}}^{e_i})$, and $\Lambda_{i} \in \varphi_{\tilde v_i}^{\tilde e_i}(g_{\tilde e_i}^q) \omega( x_{\tilde v_{i}}^{\tilde e_i}, \varphi_{\tilde v_i}^{\tilde e_i}(g_{\tilde e_i}) x_{\tilde v_{i}}^{\tilde e_i})$ for some $q < 0 \leq p$ (or $p < 0 \leq q)$.
  Hence $\Lambda_{i-1} \sim \Lambda_i$ would contradict Condition~\eqref{ax:funDDDomain}.
\end{proof}

\begin{cor}
 $(X, \mathcal{W}, \mathcal{B}, \mu)$ is a measured wallspace with an isometric $G$-action.
\end{cor}

\begin{proof}
Lemma~\ref{lem:separatingWallsAreMeasurableSet} shows that $(X, \mathcal{W}, \mathcal{B}, \mu)$ is a measured wallspace.
\end{proof}

\begin{prop} \label{prop:actionIsMetricallyProper}
 The action of $G$ on $(X, \mathcal{W}, \mathcal{B}, \mu)$ is metrically proper.
\end{prop}

\begin{proof}
 Let $x\in X$.
 Suppose that for some infinite subset $J\subset G$ and some $d>0$, we have $\#(g_ix,g_jx)<d$ for all $g_i,g_j\in J$.
 Let $S\subset T$ be the smallest subtree whose vertex spaces contain the points $\{gx: g\in J\}$.
 Note that $S$ has diameter~$\leq d$ since each edge has a vertical wall with measure $\geq 1$.

 We now verify that $S$ is locally finite.
 Indeed, if there were infinitely many edges at some $\tilde v\in S^0$, then either there would exist an edge $e$ incident to $\tilde v$, whose $G_{\tilde v}$-orbit contains infinitely many edges of $S$ at $\tilde v$, or else there are infinitely many edges in $S$ incident to $\tilde v$ that belong to distinct $G_{\tilde v}$-orbits.

 In the former case let $\{g_1 \tilde e, g_2 \tilde e, \ldots, \}_{g_i\in G_{\tilde v}}$ be an enumeration of infinitely many edges in a $G_{\tilde v}$-orbit.
 Letting $G_{\tilde e}x_{\tilde e}^{\tilde v} \subset X_{\tilde v}$ be the basepoint orbit corresponding to $\tilde e$,
 its translate   $g_iG_{\tilde e}x_{\tilde e}^{\tilde v}  \subset X_{\tilde v}$ is the basepoint corresponding to $g_i{\tilde e}$.
 By Condition~\eqref{ax:dispersed} of Theorem~\ref{thm:main}, the sets $\{g_i G_{\tilde e} x_{\tilde e}^{\tilde v}\}_{i \in \naturals}$ are
 dispersed in
 $(X_{\tilde v}, \mathcal{W}_{\tilde v}, \mathcal{B}_{\tilde v}, \mu_{\tilde v})$ and hence in $(X, \mathcal{W}, \mathcal{B}, \mu)$.
 Therefore, there exists $i,j \in \mathbb{N}$ such that $\#(g_iG_{\tilde e} x_{\tilde e}^{\tilde v}, g_jG_{\tilde e} x_{\tilde e}^{\tilde v}) > d$.
 By construction $g_iG_{\tilde e}x_{\tilde e}^{\tilde v} = G_{g_i \tilde e}x^{\tilde v}_{g_i \tilde e}$.
 Consequently, the elements of $Jx$ lying in vertex spaces separated by $g_i \tilde e$ and $g_j \tilde e$ must be at distance at least $d+2$ from each other.

 In the latter case, by definition of the measure on the vertical walls there must exist some edge $\tilde e$ in $S$, incident to $\tilde v$ such that the corresponding vertical wall $\Lambda_{\tilde e}$ has $\mu(\Lambda_{\tilde e}) > d$.
 This would imply that the two orbit points $g_1 \tilde x, g_2 \tilde x$ lying in vertices separated by $\tilde e$ have $\#(g_1 \tilde x, g_2 \tilde x) >d$.

 Therefore, $S$ is locally finite and $\diameter(S)<\infty$, hence $S$ has finitely many vertices.
 Thus, there exists an infinite subset $J'\subset J$ such that the elements of $\{g_j \tilde x : j\in J'\}$ all lie in the same vertex space $X_{\tilde v}$.
 This contradicts that $G_{\tilde v}$ acts metrically properly on $(X_{\tilde v}, \mathcal{W}_{\tilde v}, \mathcal{B}_{\tilde v}, \mu_{\tilde v})$ since $\#_{\tilde v}(p,q)=\#(p,q)$ for $p,q\in X_{\tilde v}$ by construction.
\end{proof}

\bibliographystyle{alpha}
%\bibliography{C:/Users/Dani/Dropbox/papers/wise}
\bibliography{wise}
\end{document}